\documentclass[12pt]{amsart}
\usepackage{latexsym,amssymb,amsmath}
\usepackage{amsmath,amsfonts}
\usepackage{epsfig}
\usepackage{graphicx}
\usepackage{verbatim}
\newtheorem{theorem}{Theorem}[section]
\newtheorem{lemma}[theorem]{Lemma}
\newtheorem{proposition}[theorem]{Proposition}

\newtheorem{definition}[theorem]{Definition}
\newtheorem{rmk}[theorem]{Remark}

\allowdisplaybreaks

\def\nn{\nonumber}

\def\N{{\mathbb N}}

\def\R{{\mathbb R}}

\def\la{\langle}
\def\ra{\rangle}

\def\les{\lesssim}
\def\1{{\bf 1}}

\def\eqnn{\begin{eqnarray*}}
\def\eeqnn{\end{eqnarray*}}
\def\eqn{\begin{eqnarray}}
\def\eeqn{\end{eqnarray}}
\newcommand{\nc}{\newcommand}
\nc{\be}{\begin{equation}}
\nc{\ee}{\end{equation}}
\nc{\ba}{\begin{eqnarray}}
\nc{\ea}{\end{eqnarray}}
\nc{\eps}{\epsilon}
\def\prf{\begin{proof}}
\def\endprf{\end{proof}}
\begin{document}

\title[Regularity properties of cubic NLS on $\R^+$]{Regularity properties of the cubic nonlinear Schr\"odinger equation on the half line}

\author{{\bf M.~B.~Erdo\u gan, N.~Tzirakis}\\
University of Illinois\\
Urbana-Champaign}

\thanks{Email addresses: berdogan@math.uiuc.edu (B. Erdogan), tzirakis@math.uiuc.edu (N. Tzirakis)}\thanks{ The first author was partially supported by NSF grant DMS-1501041. The second author's work was supported by a grant from the Simons Foundation (\#355523 Nikolaos Tzirakis). }

\date{}

\begin{abstract}
In this paper we study the local and global regularity properties of the cubic nonlinear Schr\"odinger equation (NLS) on the half line with rough initial data. These properties include local and global wellposedness results, local and global smoothing results and the behavior of higher order Sobolev norms of the solutions.  In particular, we prove that the nonlinear part of the cubic NLS on the half line is smoother than the initial data. The gain in regularity coincides with the gain that was observed for the periodic cubic NLS \cite{et2} and the cubic NLS on the line \cite{erin}. We also prove that in the defocusing case the norm of the solution grows at most polynomially-in-time while in the focusing case it grows exponentially-in-time. As a byproduct of our analysis we provide a different proof of an almost sharp local wellposedness in $H^s(\R^+)$.  Sharp $L^2$ local wellposedness was obtained  in \cite{holmer} and \cite{bonaetal}. 
Our methods simplify some ideas in the wellposedness theory of initial and boundary value problems that were developed in \cite{collianderkenig, holmer,holmer1,bonaetal}.
\end{abstract}

\maketitle
\section{Introduction}
We study the following initial-boundary value problem (IBVP)
\begin{align}\label{nls}
&iu_t+u_{xx}+\lambda |u|^2u=0,\,\,\,\,x\in\R^+, t\in \R^+,\\
&u(x,0)=g(x), \,\,\,\,u(0,t)=h(t). \nn
\end{align}
Here $\lambda=\pm 1$, $g\in H^s(\R^+)$ and $h\in H^{\frac{2s+1}{4}}(\R^+)$, with the additional compatibility condition $g(0)=h(0)$ for $s>\frac12$. The compatibility condition is necessary since the solutions we are interested in are    continuous space-time functions for $s>\frac12$.  

The term that models the nonlinear effects is cubic and the equation can be focusing ($\lambda=1$) or defocusing ($\lambda=-1$). Nonlinear Schr\"odinger equations (NLS) of this form model a variety of physical phenomena in optics, water wave theory and Langmuir waves in a hot plasma. In the case of the semi infinite strip $(0, \infty)\times [0,T] $, the solution $u(x,t)$ of \eqref{nls} models the amplitude of the wave generated at one end and propagating freely at the other.  For an interesting example of such a wave train in deep water waves, see \cite{ablo}.

Our intention is to study this problem by using the tools that are available to us in the case of the full line. In this case the equation is strongly dispersive, and it has been studied extensively during the last 40 years. We use the restricted norm method (also known as   the $X^{s,b}$ method) of Bourgain, \cite{bourgain, Bou2},  modified appropriately. The idea to use the restricted norm method in the case of IBVP with mild nonlinearities comes from \cite{collianderkenig}. Their paper introduced a method to solve initial-boundary value problems for nonlinear
dispersive partial differential equations by recasting these problems as initial
value problems with an appropriate forcing term. This reformulation transports the
robust theory of initial value problems to the initial-boundary value setting. The problem they considered was the Korteweg-de Vries equation on the half line. In this case to recover the derivative in the nonlinearity one has to use the cancelations of the nonlinear waves that are nicely captured by the $X^{s,b}$ method. The idea of reformulating the problem as an initial value problem with forcing was applied in the case of the NLS with a general power nonlinearity in \cite{holmer, holmer1}. The difference is that  one has to use  Strichartz estimates which are appropriate for dispersive equations with power type nonlinearities. For NLS on $\R^n$ the Strichartz estimates give sharp wellposedness results. 
One can also use more standard Laplace transform techniques to study \eqref{nls}, see e.g. \cite{bonaetal}. This is based on an explicit solution formula of the linear nonhomogeneous boundary value problem
\begin{align}\label{eq:Wb}
&iu_t+u_{xx}=0,\,\,\,\,x\in\R^+, t\in \R^+,\\
&u(x,0)=0, \,\,\,\,u(0,t)=h(t). \nn
\end{align}
which is obtain by formally using the Laplace transform. One then can use Duhamel's formula and express the nonlinear solution as a superposition of the linear evolution which incorporates the boundary term and the initial data with the nonlinearity.

We now briefly discuss a short history of the wellposedness theory for \eqref{nls}. The reader is advised to consult \cite{bonaetal} and the references therein for a more comprehensive list of works related to \eqref{nls}.
The problem on a bounded or unbounded domain $\Omega \in \R^n$ with smooth boundary and $h=0$ has been considered in \cite{brgal} and \cite{tsutsumi}. Carolle and Bu, in \cite{cbu},  considered the equation \eqref{nls} in the general case when $h$ is not identically zero. Using semigroup techniques and a priori estimates they showed the existence of a unique global solution for $g\in H^2(\R^+)$ and $h\in C^2(\R^+)$. The result was extended to the general power nonlinearity case by Bu \cite{bu} in the defocusing case. For the general domain $\Omega$ in the defocusing case Strauss and Bu in \cite{bstr} prove the existence of a global $H^1$ solution if the boundary data are smooth and compactly supported. Bu, Tsutaya, and Zhang \cite{butsuyoung} extended the above result to the focusing case in higher dimensions  for a nonlinearity of the form $|u|^{p-2}u$ and $2\leq p\leq 2+\frac{2}{n}$. For rough initial data and in the case of the half line Holmer in \cite{holmer} proved sharp local well posednesss matching the theory in the full line.  Bona, Sun, and  Zhang, \cite{bonaetal}, addressed some of the uniqueness questions that were left open in \cite{holmer}, and also studied the equation on bounded intervals. Finally, we should also mention  that in the integrable case ($p=4$) Fokas, \cite{fokas}, obtained a solution of \eqref{nls} when $g$ is a Schwartz function and $h$ is sufficiently
smooth by reformulating the problem as a $2\times 2$ matrix Riemann-Hilbert problem.

In this paper we combine the Laplace transform method \cite{bonaetal} with the  $X^{s,b}$ method \cite{bourgain} to prove that the nonlinear part of the solution is smoother than the initial data. More precisely, we prove 
\begin{theorem}\label{thm:smooth} Fix $s\in (0,\frac52)$, $s\neq \frac12, \frac32$,  $ g\in H^s(\R^+)$, and $h\in H^{\frac{2s+1}{4}}(\R^+)$, with the additional compatibility condition $g(0)=h(0)$ for $s>\frac12$. Then, for  $t$ in the local existence interval $ [0,T]$ and $a<\min(2s,\frac12,\frac52-s)$ we have
$$
u(x,t)- W_0^t(g,h)\in C^0_tH^{s+a}_x([0,T]\times \R^+),
$$
where $ W_0^t(g,h)$ is the solution of the corresponding linear equation  \eqref{nls} with  $\lambda =0$.
\end{theorem}
We note that smoothing results of this type were first obtained by Bourgain, see \cite{Bbook}, for the cubic NLS on $\R^2$. Also see \cite{kervar} for an extension of this result to $\R^n$, and \cite{erin} for cubic NLS on $\R$. There are also smoothing estimates on the torus, see, e.g., \cite{chr}, \cite{et1}, \cite{kst1}, \cite{et2}, \cite{det}, \cite{kst2}. For initial-boundary value problems, it appears that Theorem~\ref{thm:smooth} is the first smoothing result. We note that the gain in regularity matches the gain for cubic NLS both on the torus  \cite{et2} and the real line \cite{erin}. However, for defocusing NLS on the torus the smoothing gain can be improved for integer $s\geq 2$  using complete integrability methods, see \cite{kst2}.  

As an application of Theorem~\ref{thm:smooth} and a priori estimates at the energy level (see Section~\ref{sec:appendix} and \cite{bonaetal}), we obtain bounds on higher order Sobolev norms:
\begin{theorem}\label{thm:growth} In the case $s\in[1,\frac52)$, $s\neq \frac32$, $g\in H^s(\R^+)$, and $h\in H^{\frac{2s+1}4}(\R^+)\cap H^1(\R^+)$, the solution $u$ is global
and the smoothing statement holds for all times. Moreover, in the defocusing case $\|u\|_{H^s(\R^+)}$ grows at most polynomially, whereas in the focusing case it grows at most exponentially.
\end{theorem}
We note that one cannot expect to obtain better than exponential bounds in the focusing case since  the energy estimates at $H^1$ level give only  exponential bounds, see Section~\ref{sec:appendix}.

 To prove Theorem~\ref{thm:smooth}   we   combine a variety of estimates for the boundary operator (solving \eqref{eq:Wb}) and the nonlinearity in $X^{s,b}$ spaces. We also need to establish the wellposed theory in $X^{s,b}$ spaces. We obtain the solution of \eqref{nls} as a fixed point of \eqref{eq:gamma} in $X^{s,b}(\R\times [0,T])$ after extending the initial data $g$ to $\R$. As such these are limits of smooth solutions, and in particular they are mild solutions as defined in \cite{holmer}. Note that by the fixed point argument the solution is the unique solution of \eqref{eq:gamma}, however it is not a priori clear whether its restriction to $\R^+$ is independent of the extension of $g$. We resolve this issue in Section~\ref{sec:unique}.

We now discuss briefly the organization of the paper.   In Section~\ref{sec:defin} we define the notion of a solution. For 
$g\in H^s(\R^+)$ and $h\in H^{\frac{2s+1}{4}}(\R^+)$, with the additional compatibility condition $g(0)=h(0)$ for $s>\frac12$, we are looking for a solution 
\be\label{eq:space}
u \in X^{s,b}(\R\times [0,T]) \cap C^0_tH^s_x([0,T]\times \R) \cap C^0_xH^{\frac{2s+1}{4}}_t(\R \times [0,T]).
\ee

It is a well known fact that  (see \eqref{def:xsb} below for the definition of the $X^{s,b}$ norm)
$$u \in X^{s,b}(\R\times [0,T]) \subset C^0_tH^s_x([0,T]\times \R)$$
for any $b>\frac{1}{2}$. However, to close the fixed point argument we need to take  $b<\frac{1}{2}$. For this reason we need to prove the continuity of the solution directly via additional estimates for the linear evolution $W_0^t(g,h)$ (corresponding to \eqref{nls} with $\lambda=0$). The reader should keep in mind that we estimate two distinct linear processes. One is the usual solution of the free Schr\"odinger equation with initial data $g$ which we denote by $W_{\R}g$ and the other is the linear solution, $W_0^t(0,h)$ to the IBVP \eqref{eq:Wb}. We state and prove these estimates in Section~\ref{sec:apriori}, in particular in Lemmas \ref{lem:kato} and \ref{lem:wbcont}, and Proposition \ref{prop:wbh}.  These estimates also explain the regularity level of the boundary function $h$ and the selection of the spaces that are used in order to close the fixed point argument.
In the second part of Section~\ref{sec:apriori} we prove the estimates on the nonlinear terms of \eqref{eq:duhamel}, dictated by Proposition \ref{prop:duhamelkato}. In Section~\ref{sec:prt3}, we establish the local wellposedness theory, see Theorem~\ref{thm:local}. Uniqueness of the solution is immediate in the auxiliary space \eqref{eq:space}. We also present a proof of unconditional uniqueness (uniqeness of mild solutions) in Section~\ref{sec:unique}.   In Section~\ref{sec:prt3} we also discuss the dependence of the local existence time on the norms of the initial data. The estimates on the time step are crucial when we patch together local solutions to obtain a global continuous solution. This is used in Section~\ref{sec:prt1t2}, where we prove the two main theorems of this paper, Theorems \ref{thm:smooth} and \ref{thm:growth}. Section: \ref{sec:appendix} is an appendix that presents the needed a priori estimates at the energy level. As  expected the a priori estimates in the case of the focusing case are more subtle.   We close the Appendix with a technical lemma that is used throughout the text.

\subsection{Notation}

$$
\widehat g(\xi)=\mathcal F g(\xi)=\int_{\R^n} e^{-ix\cdot\xi} g(x) dx.
$$

$$
\la \xi\ra=\sqrt{1+|\xi|^2}
$$

$$
\|g\|_{H^s}=\|g\|_{H^s(\R)}=\Big(\int_\R \la \xi\ra^{2s} |\widehat g(\xi)|^2 d\xi\Big)^{1/2}
$$

For an  interval $I$, we define $H^s(I)$  norm as
$$
\|g\|_{H^s(I)}:=\inf\big\{\|\tilde g\|_{H^s(\R)}: \tilde g(x)=g(x),\, x\in I\big\}.
$$

We also denote the linear Schr\"odinger propagator (for $g\in L^2(\R)$) by
$$
W_\R g(x,t)=e^{it\Delta} g(x)= \mathcal F^{-1}\big[e^{-it|\cdot|^2} \widehat g(\cdot)\big](x).
$$

For a space time function $f$, we denote
$$
D_0f(t)=f(0,t).
$$
Finally, we  reserve the notation $\eta(t)$ for a smooth compactly supported function which is equal to $1$ on $[-1,1]$. 

\section{Notion of a solution} \label{sec:defin}

Throughout the paper we have $s\in(0,\frac52) $, $s\neq\frac12, \frac32$.
We define $H^s(\R^+)$ norm as
$$
\|g\|_{H^s(\R^+)}:=\inf\big\{\|\tilde g\|_{H^s(\R)}: \tilde g(x)=g(x),\, x>0\big\}.
$$
Note that  we  have $\|g^\prime\|_{H^{s-1}(\R^+)}\leq \|g\|_{H^{s}(\R^+)}$. 
If $g  \in  H^s(\R^+)$ for some $s>\frac12$, take an extension $\tilde g\in H^s(\R)$. By Sobolev embedding    $\tilde g $ is continuous on $\R$, and hence $g(0)$ is well defined. 
We have the following lemma concerning extensions of $H^s(\R^+)$ functions. 
\begin{lemma}\label{lem:Hs0} Let $h\in H^s(\R^+)$ for some $-\frac12< s<\frac52$. \\
i) If $-\frac12< s<\frac12$, then $\|\chi_{(0,\infty)}h\|_{H^s(\R)}\les \|h\|_{H^s(\R^+)}$.\\
ii) If $\frac12<s<\frac32$  and $h(0)=0$, then $\|\chi_{(0,\infty)}h\|_{H^s(\R)}\les \|h\|_{H^s(\R^+)}$.\\
iii) If $\frac12<s<\frac32$, then $\| h_{even}\|_{H^s(\R)}\les \|h\|_{H^s(\R^+)}$.\\
iv) If $\frac12 <s<\frac52$, $s\neq \frac32$, and $h(0)=0$, then $ \|h_{odd}\|_{H^s(\R)}\les \|h\|_{H^s(\R^+)}$.\\
Here $h_{even}(x)=h(|x|)$, and $h_{odd}$ is defined analogously.
\end{lemma}
The first two parts were proved in \cite{collianderkenig}. Part i) follows from the weighted $L^2$ boundedness of Hilbert transform and the fact that $\la\xi\ra^{2s}$ is an $A_2$ weight for $s\in (-\frac12,\frac12)$. For ii) note that, since $h(0)=0$,  the distributional derivative of $\chi_{(0,\infty)} h$ is $\chi_{(0,\infty)} h^\prime$, and use i). Part iii) follows from part ii) as follows: let $\widetilde{h}$ be an $H^s(\R)$ extension of $h$ with $\|\widetilde h\|_{H^s}\les \|h\|_{H^s(\R^+)}$. Let $f(x)=[\widetilde h(x)+\widetilde h(-x)]/2$. Note that $\|f\|_{H^s(\R)}\les \|h\|_{H^s(\R^+)}$.
Since $f(0)=h(0)$, by part ii), we have $\|(h-f)_{even}\|_{ H^s(\R)}\les \|h-f\|_{H^s(\R^+)}\les \|h\|_{H^s(\R^+)}$. Since $f_{even}=f$, we have iii). To obtain iv) note that it follows from ii) for $\frac12<s<\frac32$. Since $h(0)=0$, for $\frac32<s<\frac52$, $h_{odd}$ is continuously differentiable, and   $h_{odd}^\prime=(h^\prime)_{even}$.
Thus, $\|h_{odd}\|_{\dot H^s(\R)}\les \|h_{odd}^\prime\|_{H^{s-1}(\R)}= \|(h^\prime)_{even}\|_{H^{s-1}(\R)}\les \|h^\prime\|_{H^{s-1}(\R^+)}\les 
\|h\|_{H^{s}(\R^+)}$.

To construct the solutions of \eqref{nls} we first consider the linear problem:
\begin{align}\label{linearnls}
&iu_t+u_{xx} =0,\,\,\,\,x\in\R^+, t\in \R^+,\\
&u(x,0)=g(x)\in H^s(\R^+), \,\,\,\,u(0,t)=h(t)\in H^{\frac{2s+1}4}(\R^+), \nn
\end{align}
with the compatibility condition $h(0)=g(0)$ for $s>\frac12$. Note that the uniqueness of the solutions of equation \eqref{linearnls} follows  by considering the equation with $g=h=0$ with the method of odd extension. We now construct the unique solution of \eqref{linearnls}, that we denote by    $W_{0}^t(g,h)$,  for $t\in [0,1]$. Note that
$$
W_0^t(g,h)=W_0^t(0,h-p)+W_\R(t) g_e,
$$
where $g_e$ is an $H^s$ extension of $g$ to $\R$ satisfying $\|g_e\|_{H^s(\R)}\les \|g\|_{H^s(\R^+)}$. Moreover,  $p(t)= \eta(t) [W_\R(t) g_e]\big|_{x=0}$, which is well-defined and is in  $H^{\frac{2s+1}{4}}(\R^+)$ by Lemma~\ref{lem:kato} below. The properties of the free Schr\"odinger evolution are well known.
 To understand the first summand, $W_0^t(0,h)$, consider the linear boundary value problem \eqref{eq:Wb} with 
   $h\in H^{\frac{2s+1}{4}}(\R^+)$. Moreover for $s>\frac12$ we have the compatibility condition $h(0)=0$.
Following \cite{bonaetal} we can write the solution as $ W_0^t(0, h)=W_{1}h+W_{2}  h $, where
\begin{align}\label{eq:w1}
W_{1}h(x,t)&=\frac1\pi\int_{0}^\infty e^{-i\beta^2t+i\beta x}\beta \widehat  h(-\beta^2) d\beta,\\
W_{2}h(x,t)&=\frac1\pi \int_{0}^\infty e^{i\beta^2 t-\beta x} \beta \widehat h(\beta^2) d\beta.
\end{align}
Here by a slight abuse of notation
\be\label{eq:halffourier}
\widehat h(\xi)=\mathcal F\big(\chi_{(0,\infty)} h\big)(\xi)=\int_0^\infty e^{-i\xi t } h(t) dt.
\ee
We refer the reader the reader to \cite{bonaetal} for the derivation of $W_1$ and $W_2$. Formally, one can check that both $W_1$ and $W_2$ satisfy \eqref{eq:Wb} by differentiation. The boundary condition at $x=0$ can be justified by Fourier inversion. To see that the initial condition at $t=0$ is satisfied  apply a complex change of variable  in $W_2$ noting that $\widehat h$ is analytic in the lower half plane, for the details see \cite{bonaetal}.

By a change of variable and   Lemma~\ref{lem:Hs0}, under the conditions above we have
\be\label{eq:psiineq}
\sqrt{\int_0^\infty \langle \beta\rangle^{2s} \big|\beta \widehat h(\pm \beta^2)\big|^2 d\beta } \les \|\chi_{(0,\infty)} h\|_{H^\frac{2s+1}4(\R )} \les \|h\|_{H^\frac{2s+1}4(\R^+)}.
\ee
Note that $W_{1}$ is well-defined for $x ,t \in \R$. We also extend $W_{2}$ to all $x$ by
\be\label{eq:w2}
W_{2}h(x,t) =\frac1\pi \int_{0}^\infty e^{i\beta^2 t-\beta x} \rho(\beta x) \beta \widehat h(\beta^2) d\beta,
\ee
where $\rho(x)$ is a smooth function supported on $(-2,\infty)$, and $\rho(x)=1$ for $x>0$.

Therefore the solution of \eqref{linearnls} for $t\in [0,1] $ is given by
$$
W_0^t(g,h)=W_0^t(0,h-p)+W_\R(t) g_e,\,\,\,\,p(t)= \eta(t) [W_\R(t) g_e](0).
$$
We note that $W_0^t(g,h)$ is well-defined for $x,t\in\R$, and its restriction to $\R^+\times [0,1]$ is independent of the extension $g_e$.

The following remark will be important in the proofs of Theorem~\ref{thm:smooth} and Theorem~\ref{thm:growth}.

\begin{rmk}\label{rmk:Hslinbound}
Note that $W_{0}^t(g_1,h)-W_{0}^t(g_2,h)=W_{0}^t(g_1-g_2,0)$.  One can also obtain $W_{0}^t (g,0)$  using the method of odd extension. This implies that 
$$\|W_{0}^t (g,0)\|_{H^s(\R^+)} \leq \|W_{\R}(t) g_{odd}\|_{H^s(\R)} =\|g_{odd}\|_{H^s(\R)}\les  \|g\|_{H^s(\R^+)},$$
where we used part iv) of Lemma~\ref{lem:Hs0} in the last inequality. 
\end{rmk}

 Consider the integral equation
\begin{equation}\label{eq:duhamel}
u(t)=\eta(t)W_\R(t)g_e+ \eta(t) \int_0^tW_\R(t- t^\prime)  F(u) \,d t^\prime + \eta(t) W_0^t\big(0, h-p-q  \big)(t),
\end{equation}
where
\begin{multline*}
F(u)=\eta(t/T) | u|^2  u,\,\,\,\,p(t)= \eta(t ) D_0 (W_\R g_e),\,\,\,\text{ and }\\  q(t)=\eta(t ) D_0\Big(\int_0^tW_\R(t- t^\prime)  F(u)\, d t^\prime \Big).
\end{multline*}
Here $D_0f(t)=f(0,t)$, and $g_e$ is an  $H^s$ extension of $g$ to $\R$.  
In what follows we will prove that the  integral equation \eqref{eq:duhamel} has a unique solution in a suitable Banach space on $\R\times \R$ for some $T<1$. Using the definition of the boundary operator, it is clear that the restriction of $u$ to $\R^+\times [0,T]$ satisfies  \eqref{nls} in the distributional sense. Also note that the smooth solutions of \eqref{eq:duhamel} satisfy \eqref{nls} in the classical sense.

We work with the space $X^{s,b}(\R\times\R)$  \cite{bourgain,Bou2}:
\be\label{def:xsb}
\|u\|_{X^{s,b}}=\big\| \widehat{u}(\tau,\xi)\la \xi\ra^{s} \la \tau+\xi^2\ra^{b} \big\|_{L^2_\tau L^2_\xi}.
\ee
We  recall the embedding $X^{s,b}\subset C^0_t H^{s}_x$ for $b>\frac{1}{2}$ and the following inequalities from \cite{bourgain,gtv}.
 
For any $s,b$ we have
\begin{equation}\label{eq:xs1}
\|\eta(t)W_\R g\|_{X^{s,b}}\les \|g\|_{H^s}.
\end{equation}
For any $s\in \mathbb R$,  $0\leq b_1<\frac12$, and $0\leq b_2\leq 1-b_1$, we have
\begin{equation}\label{eq:xs2}
\Big\| \eta(t) \int_0^t W_\R(t-t^\prime)  F(t^\prime ) dt^\prime \Big\|_{X^{s,b_2} }\lesssim   \|F\|_{X^{s,-b_1} }.
\end{equation}
Moreover, for $T<1$, and $-\frac12<b_1<b_2<\frac12$, we have
\begin{equation}\label{eq:xs3}
\|\eta(t/T) F \|_{X^{s,b_1}}\les T^{b_2-b_1} \|F\|_{X^{s,b_2}}.
\end{equation}

\begin{definition}
We say \eqref{nls} is locally wellposed in $H^s(\R^+)$, if for any $g\in H^s(\R^+)$ and $h\in H^{\frac{2s+1}{4}}(\R^+)$, with the additional compatibility condition $g(0)=h(0)$ for $s>\frac12$, the equation \eqref{eq:duhamel} has a unique solution in
$$
X^{s,b}(\R\times [0,T]) \cap C^0_tH^s_x([0,T]\times \R) \cap C^0_xH^{\frac{2s+1}{4}}_t(\R \times [0,T]),
$$
for any $b<\frac12$.   Moreover, if $u$ and $v$ are two such solutions coming from different extensions $g_{e1}$, $g_{e2}$, then their restriction to $[0,\infty)\times[0,T]$ are the same.  Furthermore,   if $g_n\to g$ in $H^s(\R^+)$ and $h_n\to h$ in $H^{\frac{2s+1}{4}}(\R^+)$, then $u_n\to u$ in the space above. 
\end{definition}

\begin{theorem} \label{thm:local} Fix $s\in (0, \frac52)$, $s\neq\frac12,\frac32$.  Then \eqref{nls} is locally wellposed in $H^s(\R^+)$ with
$T\approx [C+\|g\|_{H^s(\R^+)}]^{-\frac4{2s+1}}$, where the constant $C$ depends on $\|g\|_{L^2}+\|h\|_{H^{\frac{2s+1}{4}}(\R^+)}$.
\end{theorem}

\section{A priori estimates} \label{sec:apriori}

\subsection{Estimates for linear terms}

We start with the following well known Kato smoothing estimate converting space derivatives to time derivatives. This estimate justifies the choice of spaces concerning $g$, $h$ in \eqref{nls}.  We supply a proof for completeness. 

\begin{lemma}\label{lem:kato}(Kato smoothing inequality) Fix $s\geq 0$. For any $g\in H^s(\R)$, we have
$\eta(t) W_\R g\in C^0_xH^{\frac{2s+1}4}_t(\R\times \R)$, and we have
$$
\big\|\eta  W_\R g \big\|_{L^\infty_x H^{\frac{2s+1}{4}}_t }\les \|g\|_{H^s(\R)}.
$$
\end{lemma}
\begin{proof}
Note that
\begin{multline*}
\mathcal F_t \big(  \eta  W_\R g \big)(\tau)=\int \widehat\eta(\tau+\xi^2) e^{ix\xi} \widehat{g}(\xi) d\xi \\ =\int_{|\xi|<1} \widehat\eta(\tau+\xi^2) e^{ix\xi} \widehat{g}(\xi) d\xi +\int_{|\xi|\geq 1} \widehat\eta(\tau+\xi^2) e^{ix\xi} \widehat{g}(\xi) d\xi.
\end{multline*}
We estimate the contribution of the first term to  $H^{\frac{2s+1}{4}}_t$ norm by
$$
\int_{|\xi|<1} \big\|\la\tau\ra^{\frac{2s+1}{4}}\widehat\eta(\tau+\xi^2)\big\|_{L^2_\tau} |\widehat{g}(\xi)| d\xi \les \int_{|\xi|<1} |\widehat{g}(\xi)| d\xi
\les \|\widehat g\|_{L^2}\les \|g\|_{H^s}.
$$
By a change of variable, the contribution of the second term is bounded by
$$ \Big\|\int_1^\infty  \la \tau\ra^{\frac{2s+1}{4}} |\widehat\eta(\tau+\rho)| \frac{|\widehat{g}(\pm\sqrt{\rho})|}{\sqrt\rho} d\rho\Big\|_{L^2_\tau}
\les \Big\|\int_1^\infty  \la \tau+\rho\ra^{\frac{2s+1}{4}}|\widehat\eta(\tau+\rho)| \rho^{\frac{2s+1}{4}}\frac{|\widehat{g}(\pm\sqrt{\rho})|}{\sqrt\rho} d\rho\Big\|_{L^2_\tau} .
$$
By Young's inequality, we estimate this by
$$
\|\la\cdot\ra^{\frac{2s+1}{4}} \widehat{\eta} \|_{L^1} \Big\|\rho^{\frac{2s+1}{4}}\frac{\widehat{g}(\pm\sqrt{\rho})}{\sqrt\rho}\Big\|_{L^2_{\rho>1}}\les \|g\|_{H^s}.
$$
The continuity statement follows from this and the dominated convergence theorem.
\end{proof}
Lemma~\ref{lem:wbcont} and Proposition~\ref{prop:wbh} below show that the boundary operator belongs to the space \eqref{eq:space}.  
\begin{lemma}\label{lem:wbcont}  Let  $s\geq 0 $. Then for  $h$ satisfying  $\chi_{(0,\infty)}h\in H^{\frac{2s+1}{4}}(\R )$,   we have
$W_0^t(0, h) \in C^0_tH^s_x(\R\times \R)$, and $\eta(t)W_0^t(0, h) \in C^0_xH^{\frac{2s+1}4}_t(\R\times \R)$.
\end{lemma}
\begin{proof} We start with the claim $W_2h \in C^0_tH^s_x(\R\times \R)$. Let $f(x)=e^{-x}\rho(x)$. Note that $f$ is a Schwartz function. Recalling \eqref{eq:w2},  we have
$$
W_2h=\int_0^\infty f(\beta x) e^{i\beta^2 t} \beta \widehat h(\beta^2) d\beta = \int_\R f(\beta x) \mathcal F\big(e^{-it \Delta } \psi\big)(\beta) d\beta,
$$
where
$$
\widehat \psi(\beta)=\beta \widehat h(\beta^2) \chi_{[0,\infty)}(\beta).
$$
Note that by \eqref{eq:halffourier} and  \eqref{eq:psiineq},  $\|\psi\|_{H^s}\les \|\chi_{(0,\infty)} h\|_{H^{\frac{2s+1}{4}}(\R)}$.   Using this and the continuity of $e^{-it \Delta }$ in $H^s$, it suffices to prove that
$$
Tg(x):= \int_\R f(\beta x) \widehat{g}(\beta)d\beta
$$
is bounded in $H^s$ for  $s\geq 0$. This follows from the case $s=0$ noting that
$$
\partial_x^s Tg(x)=\int_\R f^{(s)}(\beta x) \beta^s \widehat{g}(\beta)d\beta,\,\,\,\,\,s\in\N,
$$
and by interpolation.
For $s=0$, after the change of variable $\beta x\to \beta$, we have
$$
Tg(x)=\int_\R f(\beta ) x^{-1} \widehat{g}(\beta x^{-1})d\beta.
$$
Therefore,
$$
\|Tg\|_{L^2}\leq \int_\R |f(\beta )| \big\|   x^{-1} \widehat{g}(\beta x^{-1}) \big\|_{L^2_x} d\beta.
$$
Noting that
$$
\big\| x^{-1} \widehat{g}(\beta x^{-1})  \big\|_{L^2_x}^2=\int_\R x^{-2} |\widehat{g}(\beta x^{-1})|^2 dx = \int_\R \beta^{-1}   |\widehat{g}(y)|^2 dy = \beta^{-1}\|g\|_{L^2}^2,
$$
we obtain
$$
\|Tg\|_{L^2}\leq \|g\|_{L^2} \int_\R |f(\beta)| \frac{d\beta}{\sqrt{\beta}} \les \|g\|_{L^2},
$$
since $f\in\mathcal S$. This proves that $W_2h \in C^0_tH^s_x(\R\times \R)$.

To prove that $\eta(t)W_2h \in C^0_xH^{\frac{2s+1}4}_t(\R\times \R)$, write
\begin{align*}
W_2h  = \int_\R f(\beta x) \mathcal F\big(e^{-it \Delta } \psi\big)(\beta) d\beta = \int_\R \frac1x \widehat f(\xi/x)  (e^{-it \Delta } \psi) (\xi) d\xi
=  \int_\R   \widehat f(\xi )  (e^{-it \Delta } \psi) (x \xi) d\xi.
\end{align*}
The claim follows from the using Kato smoothing and dominated convergence theorem noting that $\widehat f \in L^1$.

Finally, note that 
\be\label{eq:w1psi}
W_1h=W_\R \psi,
\ee 
where
$$
\widehat \psi(\beta)=\beta \widehat h(-\beta^2) \chi_{[0,\infty)}(\beta).
$$
The claim follows  as above from \eqref{eq:halffourier}, \eqref{eq:psiineq}, the continuity of $W_\R(t)$, and Kato smoothing Lemma~\ref{lem:kato}.
\end{proof}

\begin{proposition} \label{prop:wbh} Let $b\leq \frac12$ and $s\geq 0 $. Then for  $h$ satisfying  $\chi_{(0,\infty)}h\in H^{\frac{2s+1}{4}}(\R )$,   we have
$$\|\eta(t) W_0^t(0,h) \|_{X^{s,b}} \les \|\chi_{(0,\infty)}h\|_{H_t^{\frac{2s+1}4}(\R )}.$$
\end{proposition}
 \begin{proof}  
As before, define $\psi$ as
$$
\widehat \psi(\beta)= \beta \widehat  h(-\beta^2) \chi_{(0,\infty)}(\beta).
$$
Using \eqref{eq:w1psi}, \eqref{eq:xs1}, \eqref{eq:halffourier}, and \eqref{eq:psiineq}, we have
$$
\|\eta W_{1}h\|_{X^{s,b}}= \| \eta W_\R(t)\psi \|_{X^{s,b}}  \les \| \psi \|_{H^s} \les \|\chi_{(0,\infty)}h\|_{H_t^{\frac{2s+1}4}(\R )}.
$$

For $W_2$,  by interpolation, it suffices to prove the statement for $s=0,1,2,...$. Let $f(x)=e^{-x}\rho(x)$. Note that
$$
\partial^s_x \eta W_{2}h = \eta \int_{0}^\infty e^{i\beta^2 t } f^{(s)}(\beta x)  \beta^{s+1} \widehat h(\beta^2) d\beta.
$$
Therefore, it suffices to prove the inequality for $s=0$ and $b=\frac12$.
We have
$$
\widehat{\eta W_{2}h}(\xi,\tau)=\int_{0}^\infty \widehat\eta(\tau-\beta^2)  \widehat{f}(\xi/\beta) \widehat h(\beta^2) d\beta.
$$
Since $f$ is a Schwartz function, we have
$$
\big|  \widehat{f}(\xi/\beta)\big|\les  \frac1{1+\xi^2/\beta^2}= \frac{\beta^2}{\beta^2+\xi^2}.
$$
Therefore
$$\|\eta W_{2}h\|_{X^{0,\frac12}} \les \Big\|\la \tau+\xi^2\ra^{\frac12} \int_{0}^\infty |\widehat\eta(\tau-\beta^2)| \frac{\beta^2}{\beta^2+\xi^2}| \widehat h(\beta^2)| d\beta \Big\|_{L^2_\xi L^2_\tau}.
$$
We divide this integral into pieces $\xi^2+\beta^2>1$ and  $\xi^2+\beta^2\leq 1$.
In the former case using $|\widehat\eta(\tau-\beta^2)|\les \la \tau-\beta^2\ra^{-3}$, $\la \tau+\xi^2\ra\les \la \tau-\beta^2\ra  \la \beta^2+\xi^2\ra$, and $\beta^2+\xi^2 \sim \la \beta^2+\xi^2 \ra$, we have the bound
$$
\Big\|  \int_{0}^\infty \la\tau-\beta^2\ra^{-2}  \frac{\beta^2}{( \beta^2+\xi^2)^{\frac12}} |\widehat h(\beta^2)| d\beta \Big\|_{L^2_\xi L^2_\tau }.
$$
Using Minkowski's and Young's inequalities,   we have
\begin{multline*}
\les \Big\|  \int_{0}^\infty \la\tau-\beta^2\ra^{-2} \Big\| \frac{\beta^2}{( \beta^2+\xi^2)^{\frac12}} \Big\|_{L^2_\xi } |\widehat h(\beta^2)| d\beta \Big\|_{  L^2_\tau }\les \Big\|  \int_{0}^\infty \la\tau-\beta^2\ra^{-2} \beta^{\frac32} |\widehat h(\beta^2)| d\beta \Big\|_{  L^2_\tau }\\
\les \Big\|\int_{0}^\infty \la\tau-\rho\ra^{-2} \rho^{ \frac14} |\widehat h( \rho )| d\rho\Big\|_{  L^2_\tau }\les \|\la\cdot\ra^{-2}\|_{L^1} \big\|\rho^{  \frac14}\widehat{h}(\rho)\big\|_{L^2_\rho}\les \|\chi_{(0,\infty)} h\|_{H^{\frac14}(\R)}.
\end{multline*}
In the latter case, we have the bound
$$
\Big\|\la \tau \ra^{\frac12} \int_{0}^1 \la \tau \ra^{-3}  \frac{\beta^2}{\beta^2+\xi^2}|\widehat h(\beta^2)| d\beta \Big\|_{L^2_{|\xi|\leq1} L^2_\tau}.
$$
Using Minkowski's inequality for both $L^2$ norms we have
\begin{multline*}
\les \int_{0}^1    \Big\|\frac{\beta^2}{\beta^2+\xi^2} \Big\|_{L^2_{|\xi|\leq1} } |\widehat h(\beta^2)| d\beta \les \int_0^1 \beta^{ \frac12} |\widehat h(\beta^2)|d\beta \\ \les \int_0^1 \rho^{ -\frac14} |\widehat h(\rho)|d\rho \les \|\chi_{(0,\infty)} h\|_{L^2(\R)}\leq \|\chi_{(0,\infty)} h\|_{H^{\frac14}(\R)}.
\end{multline*}
In the second to last bound we used Cauchy-Schwarz inequality.
\end{proof}

\subsection{Estimates for the nonlinear term}
In this section we establish estimates for the nonlinear term in \eqref{eq:duhamel} in order to close the fixed point argument and to obtain the smoothing theorem. 
 \begin{proposition}\label{prop:duhamelkato}  For any smooth compactly supported function $\eta$, we have
\begin{align*}
\Big\|\eta  \int_0^tW_\R(t- t^\prime) F  dt^\prime  \Big\|_{C^0_xH^{\frac{2s+1}{4}}_t(\R\times \R)}\les \left\{ \begin{array}{ll} \|F\|_{X^{s,-b}}& \text{ for } 0\leq s \leq \frac12,  b<\frac12, \\
\|F\|_{X^{\frac12,\frac{2s-1-4b}{4}}} +\|F\|_{X^{s,-b}} & \text{ for }  \frac12 \leq s \leq \frac52, b<\frac12. \end{array}
\right.
\end{align*}
\end{proposition}
\begin{proof} The proof is based on an argument from \cite{collianderkenig}. 

It suffices to prove the bound above for $\eta D_0\big(\int_0^tW_\R(t- t^\prime) F  dt^\prime \big)$ since $X^{s,b}$ norm is independent of space translation. The continuity in $x$ follows from this by dominated convergence theorem as in the proof of Lemma~\ref{lem:kato}. First we consider the case $0\leq s\leq \frac12$. Note that
$$
D_0\Big(\int_0^tW_\R(t- t^\prime) F  dt^\prime \Big)= \int_\R\int_0^t e^{-i(t-t^\prime)\xi^2}F(\widehat\xi,t^\prime) dt^\prime d\xi.
$$
Using
$$
F(\widehat\xi,t^\prime)=\int_\R e^{it^\prime\lambda}\widehat F(\xi,\lambda) d\lambda,
$$
and
$$
\int_0^t e^{it^\prime(\xi^2+\lambda)}dt^\prime = \frac{e^{it(\xi^2+\lambda)}-1}{i(\lambda+\xi^2)}
$$
we obtain
$$
D_0\Big(\int_0^tW_\R(t- t^\prime) F  dt^\prime \Big) = \int_{\R^2} \frac{e^{it \lambda }-e^{-it\xi^2}}{i(\lambda+\xi^2)} \widehat F(\xi,\lambda) d\xi d\lambda.
$$
Let $\psi$ be a smooth cutoff for $[-1,1]$, and let $\psi^c=1-\psi$. We write
\begin{multline*}
 \eta(t) D_0\Big(\int_0^tW_\R(t- t^\prime) F  dt^\prime \Big)=   \eta(t)  \int_{\R^2} \frac{e^{it \lambda }-e^{-it\xi^2}}{i(\lambda+\xi^2)} \psi(\lambda+\xi^2) \widehat F(\xi,\lambda) d\xi d\lambda \\ +\eta(t) \int_{\R^2} \frac{e^{it \lambda } }{i(\lambda+\xi^2)} \psi^c(\lambda+\xi^2) \widehat F(\xi,\lambda) d\xi d\lambda
-\eta(t) \int_{\R^2} \frac{ e^{-it\xi^2}}{i(\lambda+\xi^2)} \psi^c(\lambda+\xi^2) \widehat F(\xi,\lambda) d\xi d\lambda \\ =:I+II+III.
\end{multline*}
By Taylor expansion, we have
$$
  \frac{e^{it \lambda }-e^{-it\xi^2}}{i(\lambda+\xi^2)} =ie^{it\lambda} \sum_{k=1}^\infty \frac{(-it)^k}{k!} (\lambda+\xi^2)^{k-1}
$$
Therefore, we have
\begin{multline*}
\|I\|_{H^{\frac{2s+1}{4}}(\R)}\les  \sum_{k=1}^\infty  \frac{\| \eta(t)t^k\|_{H^1}}{k!}  \Big\|  \int_{\R^2}  e^{it\lambda} (\lambda+\xi^2)^{k-1}  \psi(\lambda+\xi^2) \widehat F(\xi,\lambda) d\xi d\lambda\Big\|_{H_t^{\frac{2s+1}{4}}(\R)}\\
\les  \sum_{k=1}^\infty  \frac{1}{(k-1)!}  \Big\| \la \lambda\ra^{\frac{2s+1}{4}} \int_{\R }    (\lambda+\xi^2)^{k-1}  \psi(\lambda+\xi^2) \widehat F(\xi,\lambda) d\xi  \Big\|_{L^2_\lambda}\\
\les  \Big\| \la \lambda\ra^{\frac{2s+1}{4}} \int_{\R }       \psi(\lambda+\xi^2) |\widehat F(\xi,\lambda)| d\xi  \Big\|_{L^2_\lambda}.
\end{multline*}
By Cauchy-Schwarz inequality in $\xi$, we estimate this by
\begin{multline*}
\Big[  \int_{\R }  \la \lambda\ra^{\frac{2s+1}{2}} \Big(  \int_{|\lambda+\xi^2|<1} \la \xi\ra^{-2s} d\xi\Big)\Big(  \int_{|\lambda+\xi^2|<1}  \la \xi\ra^{2s} |\widehat F(\xi,\lambda)|^2 d\xi\Big) d\lambda \Big]^{1/2}\\
\les \|F\|_{X^{s,-b}} \sup_\lambda   \Big( \la \lambda\ra^{\frac{2s+1}{2}} \int_{|\lambda+\xi^2|<1} \la \xi\ra^{-2s} d\xi\Big)^{1/2}\les \|F\|_{X^{s,-b}} .
\end{multline*}
 The last inequality follows by a calculation substituting $\rho =\xi^2$.

For the second term, we have
\begin{multline*}
\|II\|_{H^{\frac{2s+1}{4}}(\R)}\les \|\eta\|_{H^1} \Big\| \la \lambda\ra^{\frac{2s+1}{4}}\int_{\R } \frac{1}{ \lambda+\xi^2 } \psi^c(\lambda+\xi^2) \widehat F(\xi,\lambda) d\xi  \Big\|_{L^2_\lambda}\\
\les  \Big\| \la \lambda\ra^{\frac{2s+1}{4}}\int_{\R } \frac{1}{ \la \lambda+\xi^2 \ra } | \widehat F(\xi,\lambda)| d\xi  \Big\|_{L^2_\lambda}.
\end{multline*}
By Cauchy-Schwarz inequality in $\xi$, we estimate this by
\begin{multline*}
\Big[  \int_{\R }  \la \lambda\ra^{\frac{2s+1}{2}} \Big(  \int \frac{1}{\la \lambda+\xi^2 \ra^{2-2b} \la \xi\ra^{2s}} d\xi\Big)\Big(  \int \frac{ \la \xi\ra^{2s}}{\la \lambda+\xi^2 \ra^{2b} } |\widehat F(\xi,\lambda)|^2 d\xi\Big) d\lambda \Big]^{1/2}\\
\les \|F\|_{X^{s,-b}} \sup_\lambda   \Big( \la \lambda\ra^{\frac{2s+1}{2}}  \int \frac{1}{\la \lambda+\xi^2 \ra^{2-2b} \la \xi\ra^{2s}} d\xi\Big)^{1/2}\les \|F\|_{X^{s,-b}} .
\end{multline*}
To obtain the last inequality recall that $s\leq\frac12, b<\frac12$, and  consider the cases $|\xi|<1$ and $|\xi|\geq 1$ separately. In the former case use $\la \lambda+\xi^2 \ra\sim\la\lambda\ra$, and in the latter case use Lemma~\ref{lem:sums} after the change of variable $\rho=\xi^2$.

To estimate $\|III\|_{H^{\frac{2s+1}{4}}(\R)}$, we divide the $\xi$ integral into two pieces, $|\xi|\geq 1$, $|\xi|<1$. We estimate the contribution of the former piece as above (after the change of variable $\rho=\xi^2$):
$$
   \Big\| \la \rho\ra^{\frac{2s+1}{4}}\int_{\R } \frac{1}{ \lambda+\rho } \psi^c(\lambda+\rho) \widehat F(\sqrt{\rho},\lambda) \frac{d\lambda}{\sqrt{\rho}}  \Big\|_{L^2_{|\rho|\geq 1}}
\les  \Big\| \la \rho\ra^{\frac{2s-1}{4}}\int_{\R } \frac{1}{ \la \lambda+\rho \ra } | \widehat F(\sqrt{\rho},\lambda)| \ d\lambda   \Big\|_{L^2_{|\rho|\geq 1}}.
$$
By Cauchy-Schwarz in $\lambda$ integral, and using $b<\frac12$, we bound this by
$$
   \Big[\int_{|\rho|>1}\int_{\R } \frac{ \la \rho\ra^{\frac{2s-1}{2}}}{ \la \lambda+\rho \ra^{2b} } | \widehat F(\sqrt{\rho},\lambda)|^2   d\lambda  d\rho  \Big]^{1/2}\les \|F\|_{X^{s,-b}}.
$$
We estimate the contribution of the latter term by
$$
 \int_{\R^2} \frac{ \|\eta(t) e^{-it\xi^2}\|_{H^{\frac{2s+1}{4}}}\chi_{[-1,1]}(\xi)}{|\lambda+\xi^2|} \psi^c(\lambda+\xi^2)  |\widehat F(\xi,\lambda)| d\xi d\lambda \les \int_{\R^2} \frac{   \chi_{[-1,1]}(\xi)}{\la\lambda+\xi^2\ra}    |\widehat F(\xi,\lambda)| d\xi d\lambda.
$$
For $b<\frac12$, this is bounded by $\|F\|_{X^{0,-b}}$ by Cauchy-Schwarz inequality in $\xi$ and $\lambda$ integrals.

This finishes the proof for $0\leq s \leq \frac12$.

For $s=\frac52$, $\frac{2s+1}{4}=\frac32$, we use the inequality
$$
\|f\|_{H^\frac32}\les \|f\|_{L^2} + \|f^\prime\|_{\dot H^{\frac12}}.
$$
The required bound for the $L^2$ norm follows from the $H^\frac12$ bound above.

Note that
\begin{multline*}
\frac{d}{dt}\Big[\eta(t) D_0\Big(\int_0^tW_\R(t- t^\prime) F  dt^\prime \Big) \Big]\\ =\eta^\prime(t) D_0\Big(\int_0^tW_\R(t- t^\prime) F  dt^\prime \Big)+i \eta(t)\int_{\R^2} \frac{\lambda e^{it \lambda }+\xi^2 e^{-it\xi^2}}{  \lambda+\xi^2 } \widehat F(\xi,\lambda) d\xi d\lambda\\=
\eta^\prime(t) D_0\Big(\int_0^tW_\R(t- t^\prime) F  dt^\prime \Big) \\ +i \eta(t)\int_{\R^2} \frac{  e^{it \lambda }-  e^{-it\xi^2}}{  \lambda+\xi^2 } (-\xi^2) \widehat F(\xi,\lambda) d\xi d\lambda+i \eta(t)\int_{\R^2} \frac{ e^{it \lambda } }{  \la\lambda+\xi^2\ra } \la \lambda+\xi^2\ra \widehat F(\xi,\lambda) d\xi d\lambda.
\end{multline*}
We bound the first integral in the last line using the case $s=\frac12$ we obtained above for  $\widehat G_1(\xi,\lambda)=\xi^2 \widehat F(\xi,\lambda)$, and the second integral using the proof of the case II for $\widehat G_2(\xi,\lambda)=\la \lambda+\xi^2\ra \widehat F(\xi,\lambda)$. Thus, we obtain 
\begin{multline*}
\Big\|\frac{d}{dt}\Big[\eta(t) D_0\Big(\int_0^tW_\R(t- t^\prime) F  dt^\prime \Big) \Big]\Big\|_{H^{\frac12}} \\
\les \|F\|_{X^{\frac12,-b}}+\|G_1\|_{X^{\frac12,-b}}+\|G_2\|_{X^{\frac12,-b}}\les \|F\|_{X^{\frac12,1-b}}+\|F\|_{X^{\frac52,-b}},
\end{multline*}
for all $b<\frac12$. 

Therefore, we have
\begin{align*}
\big\|\eta D_0\Big(\int_0^tW_\R(t- t^\prime) F  dt^\prime \Big)\big\|_{H^{\frac{2s+1}{4}}(\R)}\les \left\{ \begin{array}{ll} \|F\|_{X^{s,-b}}& \text{ for } 0\leq s \leq \frac12,  b<\frac12, \\
\|F\|_{X^{\frac12,1-b}}+\|F\|_{X^{\frac52,-b}}& \text{ for }   s = \frac52, b<\frac12. \end{array}
\right.
\end{align*}
We obtain the statement for $\frac12<s<\frac52$ by interpolation.

\end{proof}

\begin{proposition}\label{prop:smooth} For fixed $s>0$ and $a<\min(2s,\frac{1}{2})$, there exists $\epsilon>0$ such that for $\frac12-\epsilon<b <\frac12$, we have
$$\big\||u|^2u\big\|_{X^{s+a,-b }}\les \|u\|_{X^{s,b }}^3.$$ 
\end{proposition}
 \begin{proof}
 By writing the Fourier transform of $|u|^2u =u\bar{u}u$ as a convolution, we obtain
\[ \widehat{|u|^2u}(\xi,\tau) = \int_{\xi_1,\xi_2}\int_{\tau_1,\tau_2} \widehat{u}(\xi_1,\tau_1)\overline{\widehat{u}(\xi_2,\tau_2)} \widehat{u}(\xi - \xi_1 + \xi_2, \tau - \tau_1+\tau_2). \]
Hence
\[
\| |u|^2u\|_{X^{s+a,-b }}^2 = \left\| \int_{\xi_1,\xi_2}\int_{\tau_1,\tau_2} \frac{\langle\xi\rangle^{s+a} \widehat{u}(\xi_1,\tau_1)\overline{\widehat{u}(\xi_2,\tau_2)} \widehat{u}(\xi - \xi_1 + \xi_2, \tau - \tau_1+\tau_2)}{\langle \tau  + \xi^2 \rangle^{b }} \right\|_{L^2_\xi L^2_\tau}^2. \]
We define
\[ f(\xi,\tau) = |\widehat{u}(\xi,\tau)|\langle \xi \rangle ^s \langle \tau+ \xi^2 \rangle^{b } \]
and
\[ M(\xi_1,\xi_2,\xi,\tau_1,\tau_2,\tau) = \frac{ \langle\xi\rangle^{s+a} \langle \xi_1 \rangle^{-s} \langle \xi_2 \rangle ^{-s} \langle \xi -\xi_1 + \xi _2 \rangle ^{-s}}{\langle \tau +\xi ^2 \rangle ^{b } \langle \tau_1 + \xi_1 ^2 \rangle ^{b } \langle \tau_2 + \xi_2 ^2 \rangle ^{b } \langle \tau - \tau_1 + \tau_2 + (\xi - \xi_1 + \xi_2 )^2 \rangle ^{b } }. \]
It is then sufficient to show that
\begin{align*} \left\| \int_{\xi_1,\xi_2}\int_{\tau_1,\tau_2} M(\xi_1,\xi_2,\xi,\tau_1,\tau_2,\tau) f(\xi_1,\tau_1)f(\xi_2,\tau_2)f(\xi-\xi_1+\xi_2,\tau-\tau_1+\tau_2) \right\|_{L^2_\xi L^2_\tau}^2 \\
\lesssim \|f\|_{L^2}^6 = \|u\|_{X^{s,b }}^6.
\end{align*}
By applying Cauchy-Schwarz in the $\xi_1,\xi_2,\tau_1,\tau_2$ integral and then using H\"{o}lder's inequality, we bound the norm above by
\begin{align*}
&\left\| \left( \int_{\xi_1,\xi_2}\int_{\tau_1,\tau_2} M^2\right)^{1/2}  \left( \int_{\xi_1,\xi_2}\int_{\tau_1,\tau_2} f^2(\xi_1,\tau_1)f^2(\xi_2,\tau_2)f^2(\xi-\xi_1+\xi_2,\tau-\tau_1+\tau_2)  \right)^{1/2} \right\|_{L^2_\xi L^2_\tau}^2 \\
&= \left\| \left( \int_{\xi_1,\xi_2}\int_{\tau_1,\tau_2} M^2\right)  \left( \int_{\xi_1,\xi_2}\int_{\tau_1,\tau_2} f^2(\xi_1,\tau_1)f^2(\xi_2,\tau_2)f^2(\xi-\xi_1+\xi_2,\tau-\tau_1+\tau_2) \right)  \right\|_{L^1_\xi L^1_\tau} \\
&\leq \sup_{\xi ,\tau} \left( \int_{\xi_1,\xi_2}\int_{\tau_1,\tau_2} M^2\right) \cdot \left\| \int_{\xi_1,\xi_2}\int_{\tau_1,\tau_2} f^2(\xi_1,\tau_1)f^2(\xi_2,\tau_2)f^2(\xi-\xi_1+\xi_2,\tau-\tau_1+\tau_2)  \right\|_{L^1_\xi L^1_\tau}\\
&= \sup_{\xi ,\tau} \left( \int_{\xi_1,\xi_2}\int_{\tau_1,\tau_2} M^2\right) \cdot \left\| f^2 \ast f^2 \ast f^2 \right\|_{L^1_\xi L^1_\tau}.
\end{align*}
Using Young's inequality, the norm $\left\| f^2 \ast f^2 \ast f^2 \right\|_{L^1_\xi L^1_\tau}$ can be estimated by $\|f\|_{L^2_\xi L^2_\tau}^6$. Thus it is sufficient to show that the supremum above is finite.
Using Lemma~\ref{lem:sums} in the $\tau_1, \tau_2$ integrals, the supremum is bounded by
\[ \sup_{\xi, \tau} \int \frac{ \langle\xi\rangle^{2s+2a} \langle \xi_1 \rangle^{-2s} \langle \xi_2 \rangle ^{-2s} \langle \xi -\xi_1 + \xi _2 \rangle ^{-2s}}{\langle \tau + \xi ^2 \rangle ^{2 b } \langle \tau+ \xi_1^2 - \xi_2^2 + (\xi - \xi_1 + \xi_2 )^2 \rangle ^{6b -2} } d\xi_1d\xi_2. \]

Using the relation $\langle \tau - a \rangle \langle \tau - b \rangle \gtrsim \langle a-b \rangle$, the above reduces to
\begin{align*}& \quad \sup_{\xi} \int \frac{ \langle\xi\rangle^{2s+2a} \langle \xi_1 \rangle^{-2s} \langle \xi_2 \rangle ^{-2s} \langle \xi -\xi_1 + \xi _2 \rangle ^{-2s}}{ \langle \xi^2 - \xi_1^2 + \xi_2^2 - (\xi - \xi_1 + \xi_2 )^2 \rangle ^{1-} } d\xi_1d\xi_2 \\
&= \sup_{\xi} \int \frac{ \langle\xi\rangle^{2s+2a} \langle \xi_1 \rangle^{-2s} \langle \xi_2 \rangle ^{-2s} \langle \xi -\xi_1 + \xi _2 \rangle ^{-2s}}{ \langle 2(\xi_1 - \xi)(\xi_1-\xi_2) \rangle ^{1-} } d\xi_1d\xi_2.
\end{align*}

We break the integral  into two pieces. The argument given in \cite{et2} shows that
$$ \sup_{\xi} \int_{\substack{ |\xi_1-\xi| \geq 1 \\ |\xi_1-\xi_2| \geq 1}} \frac{ \langle\xi\rangle^{2s+2a} \langle \xi_1 \rangle^{-2s} \langle \xi_2 \rangle ^{-2s} \langle \xi -\xi_1 + \xi _2 \rangle ^{-2s}}{ \langle (\xi_1-\xi)(\xi_1-\xi_2) \rangle ^{1-} } d\xi_1 d\xi_2 < \infty.
$$
To estimate the integral on the remaining set, $\{|\xi_1-\xi| \leq 1 \text{ or } |\xi_1-\xi_2| \leq 1\}$, note that
\be\label{eq:A}
\la \xi_1\ra \la\xi-\xi_1+\xi_2\ra \sim \la \xi_2\ra \la\xi \ra.
\ee
Therefore, we have
$$
\int_{\substack{ |\xi_1-\xi| \leq 1 \text{ or }\\ |\xi_1-\xi_2| \leq 1}}  \frac{ \langle\xi\rangle^{2s+2a} \langle \xi_1 \rangle^{-2s} \langle \xi_2 \rangle ^{-2s} \langle \xi -\xi_1 + \xi _2 \rangle ^{-2s}}{ \langle (\xi_1-\xi)(\xi_1-\xi_2) \rangle ^{1-} }  d\xi_1 d\xi_2 \les \int \frac{ \langle\xi\rangle^{ 2a}
  \langle \xi_2 \rangle ^{-4s} }{ \langle (\xi_1-\xi)(\xi_1-\xi_2) \rangle ^{1-} }  d\xi_1 d\xi_2
$$
we   use the substitution $x = (\xi_1 - \xi)(\xi_1-\xi_2)$ in the $\xi_1$ integral. This yields
$$
 2\xi_1 = \xi + \xi_2 \pm \sqrt{(\xi+\xi_2)^2 - 4(\xi\xi_2 - x)} = \xi + \xi_2 \pm \sqrt{4x + (\xi-\xi_2)^2}
 $$
 and
$$
dx = (2\xi_1-\xi-\xi_2) \; d\xi_1 = \pm \sqrt{4x + (\xi-\xi_2)^2} \; d\xi_1.
$$
Therefore, the integral above is bounded by
\[  \int \frac{ \langle\xi\rangle^{2a} \langle \xi_2 \rangle ^{-4s} }{ \langle x \rangle ^{1-} \sqrt{|4x + (\xi-\xi_2)^2|} } dx \; d\xi_2.\]

Using Lemma~\ref{lem:sqrt} and then Lemma~\ref{lem:sums} again, we bound the supremum of the integral above by

\begin{align*}
\sup_{\xi} \int \frac{\langle\xi\rangle^{2a} \langle \xi_2 \rangle ^{-4s} }{\langle (\xi - \xi_2)^2 \rangle ^{ \frac12- }} d\xi_2
&\les\sup_{\xi} \int \frac{\langle\xi\rangle^{2a} \langle \xi_2 \rangle ^{-4s} }{\langle \xi - \xi_2 \rangle ^{1-}} d\xi_2\\
&\les \sup_{\xi} \begin{cases}  \langle \xi \rangle ^{2a-1+}  &\text{for } s \geq \frac14 \\
 \langle \xi \rangle ^{2a-4s+}  &\text{for } s < \frac14 \end{cases}.
\end{align*}
For $a < \min(\frac12, 2s)$, this is finite.
 \end{proof}

 \begin{proposition}\label{prop:smooth2} For fixed $0< s <\frac52$, and  $0\leq a<\min(2s,\frac{1}{2}, \frac52-s)$,   there exists $\epsilon>0$ such that for $\frac12-\epsilon<b <\frac12$, we have
\begin{align*}
  & \text{ for } 0< s + a\leq \frac12,  \,\,\,\,\,\,\,\, \big\||u|^2u\big\|_{X^{s+a,-b}} \les \|u\|_{X^{s,b}}^3,  \\
&  \text{ for }  \frac12 < s +a < \frac52,  \,\,\,\,\,\,\,\,  \big\||u|^2u\big\|_{X^{\frac12,\frac{2s+2a-1-4b}{4}}} \les \|u\|_{X^{s,b}}^3.
\end{align*} 
\end{proposition}
\begin{proof} For $s+a\leq \frac12$, the statement follows from Proposition~\ref{prop:smooth}.

We now consider the case $\frac12 < s +a < \frac52$. Since $a<2s$,   we always have $s>\frac16$.   Let \[ S:=   \int \frac{ \langle \tau + \xi ^2 \rangle^{s+a-2b-\frac12} \langle\xi\rangle  \langle \xi_1 \rangle^{-2s} \langle \xi_2 \rangle ^{-2s} \langle \xi -\xi_1 + \xi _2 \rangle ^{-2s}}{ \langle \tau+ \xi_1^2 - \xi_2^2 + (\xi - \xi_1 + \xi_2 )^2 \rangle ^{6b -2} } d\xi_1d\xi_2. \]Following the proof of Proposition~\ref{prop:smooth}, it suffices to prove that
\[ \sup_{\xi, \tau}  S <\infty. \]
We consider the cases $\frac12<s+a<\frac32$ and $\frac32\leq s+a<\frac52$ separately.

{\bf Case 1) } $\frac12< s+a<\frac32$. Taking $\epsilon$ sufficiently small, we have $s+a-2b-\frac12<0$. Using the identity $\langle \tau - a \rangle \langle \tau - b \rangle \gtrsim \langle a-b \rangle$, and noting that $2b+\frac12-s-a<6b-2$ (for $\epsilon$ sufficiently small), we obtain
\begin{multline*}
S\les    \int \frac{  \langle\xi\rangle  \langle \xi_1 \rangle^{-2s} \langle \xi_2 \rangle ^{-2s} \langle \xi -\xi_1 + \xi _2 \rangle ^{-2s}}{ \langle \xi ^2 -\xi_1^2 + \xi_2^2 - (\xi - \xi_1 + \xi_2 )^2 \rangle ^{2b+\frac12-s-a} } d\xi_1d\xi_2 \\
 \les  \int \frac{  \langle\xi\rangle  \langle \xi_1 \rangle^{-2s} \langle \xi_2 \rangle ^{-2s} \langle \xi -\xi_1 + \xi _2 \rangle ^{-2s}}{ \langle (\xi_1-\xi) (\xi_1-\xi_2) \rangle ^{2b+\frac12-s-a} } d\xi_1d\xi_2.
\end{multline*}
We can estimate this for $s>\frac12$ by
$$
 \int \langle\xi\rangle  \langle \xi_1 \rangle^{-2s} \langle \xi_2 \rangle ^{-2s} \langle \xi -\xi_1 + \xi _2 \rangle ^{-2s}  d\xi_1d\xi_2 \les 1
$$
by using Lemma~\ref{lem:sums} twice.

It remains to consider the case $\frac16<s\leq\frac12$. Since $a<\min(2s,\frac12), $ we have $\frac12<s+a<\min(3s,s+\frac12) $.

Consider the sets $A= \{|x_1-\xi|< 1 \text{ or }  |x_1-\xi_2| < 1\}$ and $B= \{|x_1-\xi|\geq 1 \text{ and }  |x_1-\xi_2| \geq 1\}$.
Since on $A$ we have \eqref{eq:A}, we obtain
$$
\int_A \frac{  \langle\xi\rangle  \langle \xi_1 \rangle^{-2s} \langle \xi_2 \rangle ^{-2s} \langle \xi -\xi_1 + \xi _2 \rangle ^{-2s}}{ \langle (\xi_1-\xi) (\xi_1-\xi_2) \rangle ^{2b+\frac12-s-a} } d\xi_1d\xi_2 \les \int_A \frac{  \langle\xi\rangle^{1-2s}   \langle \xi_2 \rangle ^{-4s} }{ \langle (\xi_1-\xi) (\xi_1-\xi_2) \rangle ^{2b+\frac12-s-a} } d\xi_1d\xi_2.
$$
Proceeding as in Proposition~\ref{prop:smooth} by substituting $x=(\xi_1-\xi) (\xi_1-\xi_2)$ in the $\xi_1$ integral, we bound this by
$$
\int  \frac{  \langle\xi\rangle^{1-2s}   \langle \xi_2 \rangle ^{-4s} }{ \la  x \ra^{2b+\frac12-s-a} \sqrt{|4x+(\xi-\xi_2)^2|}} dx d\xi_2\les \int  \frac{  \langle\xi\rangle^{1-2s}   \langle \xi_2 \rangle ^{-4s} }{   \la\xi-\xi_2\ra^{2 (2b-s-a)}} d\xi_2,
$$
where we used Lemma~\ref{lem:sqrt} (taking $\epsilon$ sufficiently small). Using Lemma~\ref{lem:sums} (noting that $2 (2b-s-a)<1$), we bound this by
\begin{align*}
 \begin{cases}  \langle \xi \rangle ^{2-4b+2a-4s+ }  &\text{for }   s\leq \frac14 \\
  \langle \xi \rangle ^{1-4b+2a   }  &\text{for }   s>\frac14    \end{cases}
\end{align*}
which is bounded for $a<\min(2s,\frac12)$, provided that $\epsilon$ is sufficiently small.

We bound the integral on the set $B$ by (after the change of variable $\xi_2 \to \xi_1+\xi_2$, $\xi_1\to \xi+\xi_1$)
\begin{multline*}
\int \frac{  \langle\xi\rangle  \langle \xi_1 \rangle^{-2s} \langle \xi_2 \rangle ^{-2s} \langle \xi -\xi_1 + \xi _2 \rangle ^{-2s}}{ \la \xi_1-\xi\ra^{2b+\frac12-s-a}  \la \xi_1-\xi_2 \ra^{2b+\frac12-s-a} } d\xi_1d\xi_2= \\ \int \frac{  \langle\xi\rangle  \langle \xi+\xi_1 \rangle^{-2s} \langle \xi+\xi_2\rangle ^{-2s} \langle \xi +\xi_1 + \xi _2 \rangle ^{-2s}}{ \la \xi_1 \ra^{2b+\frac12-s-a}  \la  \xi_2 \ra^{2b+\frac12-s-a} } d\xi_1d\xi_2.
\end{multline*}
By symmetry, we have the following subcases $|\xi+\xi_1+\xi_2|\gtrsim |\xi|$ and  $|\xi+ \xi_1 |\gtrsim |\xi|$, which leads to the bound (using Lemma~\ref{lem:sums} repeatedly)
\begin{multline*}
\langle\xi\rangle^{1-2s}\Big( \int \frac{   \langle \xi+\xi_1 \rangle^{-2s} }{ \la \xi_1 \ra^{2b+\frac12-s-a}    } d\xi_1\Big)^2 + \langle\xi\rangle^{1-2s}
\int \frac{    \la \xi+\xi_2\ra^{-2s} \langle \xi +\xi_1 + \xi _2 \rangle ^{-2s}}{ \la \xi_1 \ra^{2b+\frac12-s-a}  \la  \xi_2 \ra^{2b+\frac12-s-a} } d\xi_1d\xi_2
\\
\les \la\xi\ra^{1-2s} \big(\la\xi\ra^{-(s+2b-\frac12-a)+}\big)^2 + \la\xi\ra^{1-2s} \int \frac{   1}{ \la \xi+\xi_2 \ra^{3s+2b-\frac12- a-}  \la  \xi_2 \ra^{2b+\frac12-s-a} }  d\xi_2 \\
\les \langle \xi \rangle ^{2-4b+2a-4s+ } + \begin{cases}  \langle \xi \rangle ^{2-4b+2a-4s+ }  &\text{for }   3s+2b-\frac12- a \leq 1 \\
  \langle \xi \rangle ^{\frac12-2b-s+a}   &\text{for }   3s+2b-\frac12- a >1    \end{cases}
\end{multline*}
This is bounded for  $a<\min(2s,\frac12)$, provided that $\epsilon$ is sufficiently small.

{\bf Case 2) } $\frac32\leq  s+a<\frac52$. In this case  $s+a-2b-\frac12\geq 0$. Using
\begin{multline*}
\la\tau+\xi^2\ra=  \la \tau+ \xi_1^2 - \xi_2^2 + (\xi - \xi_1 + \xi_2 )^2 + 2(\xi-\xi_1)(\xi_1-\xi_2)\ra\\
\les  \la \tau+ \xi_1^2 - \xi_2^2 + (\xi - \xi_1 + \xi_2 )^2 \ra + \la \xi-\xi_1\ra \la \xi_1-\xi_2\ra.
\end{multline*}
Also noting that in this case $s+a-2b-\frac12<6b-2$ for ($\epsilon$ sufficiently small), we have
\begin{multline*}
S\les \int   \la \xi-\xi_1\ra^{s+a-2b-\frac12} \la \xi_1-\xi_2\ra^{s+a-2b-\frac12} \langle\xi\rangle  \langle \xi_1 \rangle^{-2s} \langle \xi_2 \rangle ^{-2s} \langle \xi -\xi_1 + \xi _2 \rangle ^{-2s}  d\xi_1d\xi_2 \\
 =
\int   \la  \xi_1\ra^{s+a-2b-\frac12} \la  \xi_2\ra^{s+a-2b-\frac12} \langle\xi\rangle  \langle \xi+\xi_1 \rangle^{-2s} \langle \xi+\xi_2 \rangle ^{-2s} \langle \xi +\xi_1 + \xi _2 \rangle ^{-2s}  d\xi_1d\xi_2.
\end{multline*}
Here we applied the change of variable $ \xi_2\to \xi_1+\xi_2$, $  \xi_1\to \xi+\xi_1 $. Considering the subcases $|\xi+\xi_1+\xi_2|\gtrsim |\xi|$ and  $|\xi+ \xi_1 |\gtrsim |\xi|$ we have the bound
\begin{multline*}
S\les \la \xi\ra^{1-2s} \Big( \int   \la  \xi_1\ra^{s+a-2b-\frac12}    \langle \xi+\xi_1 \rangle^{-2s}    d\xi_1\Big)^2 \\
+ \la \xi\ra^{1-2s} \int   \la  \xi_1\ra^{s+a-2b-\frac12} \la  \xi_2\ra^{s+a-2b-\frac12}  \langle \xi+\xi_2 \rangle ^{-2s} \langle \xi +\xi_1 + \xi _2 \rangle ^{-2s}  d\xi_1d\xi_2\\
=:S_1+S_2.
\end{multline*}
Using $\la\xi_1\ra\les \la \xi+\xi_1\ra \la \xi\ra$, we have
$$
S_1\les  \la \xi\ra^{ 2  a-4b } \Big( \int     \langle \xi+\xi_1 \rangle^{-s+a-2b-\frac12}    d\xi_1\Big)^2 \les 1
$$
by the restrictions on $a,b,s$. Using $ \la  \xi_1\ra \les \la \xi+\xi_2\ra \la \xi+\xi_1+\xi_2\ra$ and $\la \xi_2\ra \les \la\xi\ra \la\xi+\xi_2\ra $ we have
$$
S_2\les  \la \xi\ra^{ \frac12-s+a-2b}   \int     \langle \xi+\xi_2 \rangle^{2a-4b-1}\la \xi+\xi_1+\xi_2\ra^{a-2b-\frac12-s}   d\xi_1d\xi_2 \les 1
$$
by the restrictions on $a,b,s$.
\end{proof}

\section{Local theory: The proof of  Theorem~\ref{thm:local}} \label{sec:prt3}

We first prove that
\begin{equation}\label{eq:gamma}
\Gamma u(t):=\eta(t)W_\R(t)g_e+ \eta(t) \int_0^tW_\R(t- t^\prime)  F(u) \,d t^\prime + \eta(t) W_0^t\big(0, h-p-q  \big)(t),
\end{equation}
has a fixed point in $X^{s,b}$. Here  $s\in(0,\frac52)$, $s\neq \frac12,\frac52$,    $b<\frac12$ is sufficiently close to $\frac12$, and
\begin{multline*}
F(u)=\eta(t/T) | u|^2  u,\,\,\,\,p(t)= \eta(t ) D_0 (W_\R g_e),\,\,\,\text{ and }\\  q(t)=\eta(t ) D_0\Big(\int_0^tW_\R(t- t^\prime)  F(u)\, d t^\prime \Big).
\end{multline*}

To see that $\Gamma$ is bounded in $X^{s,b}$ recall the following bounds:

By \eqref{eq:xs1}, we have
$$
\|\eta W_\R (t) g_e\|_{X^{s,b}} \les \|g_e\|_{H^s}\les \|g\|_{H^s(\R^+)}.
$$
Combining \eqref{eq:xs2}, \eqref{eq:xs3}, and Proposition~\ref{prop:smooth}, we obtain
\begin{multline*}
  \| \eta(t) \int_0^tW_\R(t- t^\prime)  F(u) \,d t^\prime \|_{X^{s,b}} \les \|F(u)\|_{X^{s,-\frac12+}} \les T^{\frac12-b-} \||u|^2u\|_{X^{s,-b}}\les T^{\frac12-b-} \|u\|_{X^{s,b}}^3.
\end{multline*}
Using Proposition~\ref{prop:wbh} and Lemma~\ref{lem:Hs0} (noting that the compatibility condition holds) we have
\begin{multline} \label{eq:temp1}
 \| \eta(t) W_0^t\big( 0, h-p-q \big)(t)\|_{X^{s,b}}
\les \|(h-p-q)\chi_{(0,\infty)} \|_{H^{\frac{2s+1}{4}}_t(\R)} \\ \les \|h-p\|_{H^{\frac{2s+1}{4}}_t(\R^+)} + \|q\|_{H^{\frac{2s+1}{4}}_t(\R^+)}
\les \|h \|_{H^{\frac{2s+1}{4}}_t(\R^+)} +\|p\|_{H^{\frac{2s+1}{4}}_t(\R )}+\|q\|_{H^{\frac{2s+1}{4}}_t(\R )}.
\end{multline}
By Kato smoothing Lemma~\ref{lem:kato}, we have
$$
\|p\|_{H^{\frac{2s+1}{4}}_t(\R )} \les \|g\|_{H^s(\R^+)}.
$$
Finally, by  Propostion~\ref{prop:duhamelkato}, \eqref{eq:xs3}, and  Proposition~\ref{prop:smooth2} we have
\begin{multline*}
\|q\|_{H^{\frac{2s+1}{4}}_t(\R )} \les \left\{ \begin{array}{ll} \|F\|_{X^{s,-\frac12+}}& \text{ for } 0\leq s \leq \frac12   \\
\|F\|_{X^{\frac12,\frac{2s-3+}{4}}}+\|F\|_{X^{s,-\frac12+}}& \text{ for }   \frac12<  s < \frac52 \end{array}
\right. \\ \les T^{ \frac12-b-}\left\{ \begin{array}{ll} \||u|^2u\|_{X^{s,-b}}& \text{ for } 0\leq s \leq \frac12  \\
\||u|^2u\|_{X^{\frac12,\frac{2s-1-4b}{4}}}+\||u|^2u\|_{X^{s,-b}}& \text{ for }  \frac12<  s < \frac52 \end{array}\right.
\les T^{ \frac12-b-} \| u\|_{X^{s,b}}^3.
\end{multline*}
Combining these estimates, we obtain
$$
\|\Gamma u\|_{X^{s,b}}\les \|g\|_{H^s(\R^+)}+ \|h \|_{H^{\frac{2s+1}{4}}_t(\R^+)} + T^{ \frac12-b-} \| u\|_{X^{s,b}}^3.
$$
This yields the existence of a fixed point $u$ in $X^{s,b}$.
Now we prove that $u\in C^0_t H^s_x([0,T)\times \R)$. Note that the first term in the definition \eqref{eq:gamma}  is continuous in $H^s$.  The continuity of the third term
follows from Lemma~\ref{lem:wbcont}  and \eqref{eq:temp1}. For the second term it follows from the embedding  $X^{s,\frac12+}\subset C^0_tH^s_x$
 and \eqref{eq:xs2} together with Proposition~\ref{prop:smooth}. The fact that  $u\in C^0_x H^{\frac{2s+1}{4}}_t(\R\times [0,T])$ follows similarly from Lemma~\ref{lem:kato}, Proposition~\ref{prop:duhamelkato}, and Lemma~\ref{lem:wbcont}.

The continuous dependence on the initial and boundary data follows from the fixed point argument and the a priori estimates as in the previous paragraph.
The uniqueness issue is discussed in Section~\ref{sec:unique} below. 

To finish the proof of Theorem~\ref{thm:local} we need to quantify the dependence of $T$ to initial and boundary data. By scaling, note that if $u$ solves the equation on $[0,\lambda^{-2}]$, then $u^\lambda(x,t)=\frac1\lambda u(\frac x\lambda,\frac t{\lambda^2})$ solves the equation with data $g^\lambda(x)=\frac1\lambda g(\frac x\lambda)$ and $h^\lambda(t)=\frac1\lambda h(\frac t{\lambda^2})$ on $[0,1]$. Note that for $\lambda>1$,
$$
\|h^\lambda\|_{H^{\frac{2s+1}{4}}(\R^+)}\les \|h\|_{H^{\frac{2s+1}{4}}(\R^+)},
$$
\begin{multline*} \|g^\lambda\|_{H^s(\R^+)}\leq \| g_\lambda\|_{L^2(\R^+)}+ \|g^\lambda\|_{\dot H^s(\R^+)}\\ \leq \lambda^{-1/2} \|g\|_{L^2(\R^+)} + \lambda^{-\frac12-s} \|g \|_{\dot H^s(\R^+)} \leq \|g\|_{L^2(\R^+)} + \lambda^{-\frac12-s} \|g \|_{H^s(\R^+)}
\end{multline*}
Therefore, for $ \lambda^{-\frac12-s} \|g \|_{H^s(\R^+)} \approx 1$, the solution in $[0,\lambda^{-2}] $ is defined up  to  the local existence time $T\approx [C+\|g\|_{H^s(\R^+)}]^{-\frac4{2s+1}}$. Here the constant $C$ depends on $\|g\|_{L^2}+\|h\|_{H^{\frac{2s+1}{4}}(\R^+)}$.
Alternatively, to obtain a local existence interval without implicit dependence on $\|g\|_{L^2}$  we can use
$$
 \|g^\lambda\|_{H^s}\leq \| g_\lambda\|_{L^2}+ \|g^\lambda\|_{\dot H^s}\leq \lambda^{-1/2} \|g\|_{L^2} + \lambda^{-\frac12-s} \|g \|_{\dot H^s}
 \leq  \lambda^{-\frac12 } \|g \|_{ H^s},
$$
which leads to $T\approx [C+\|g\|_{  H^s}]^{-4}$ with $C=C(\|h\|_{H^{\frac{2s+1}{4}}(\R^+)})$. This will be used in Section~\ref{sec:prt1t2} below.

\subsection{Uniqueness of mild solutions}\label{sec:unique} In this section we discuss the uniqueness of solutions of \eqref{nls}, also see \cite{holmer, bonaetal}. 
The solution we constructed above is the unique fixed point of \eqref{eq:gamma}. However, it is not a priori clear if different extensions of initial data produce  the same solution on $\R^+$.  It is also not clear if the solution we constructed is same as the solutions obtained in \cite{holmer, bonaetal}. 
To resolve this issue,  
first note that the restriction to $\R^+$ of the solution we constructed (the fixed point of \eqref{eq:gamma}) is a mild solution as defined in \cite[Definition 3]{holmer}. Therefore, the uniqueness part of the theorem for $s>\frac12$ follows from a simple argument based on energy estimates which implies the uniqueness of mild solutions, see  \cite[Proposition~1]{holmer}. In particular, the restriction of $u$ to $\R^+$ is independent of the extension $g_e$ of $g$ to $\R^+$.

We now prove that the restriction of $u$ to $\R^+$ is independent of the extension of $g$ also in the case $s\in(0,1/2)$. Let $g_1$, $g_2$ be two $H^s(\R)$ extensions of $g\in H^s(\R^+)$, $s\in(0,1/2)$. Take a sequence $f_n \in H^2(\R^+)$ converging to $g$ in $H^s(\R^+)$.
Let $f_n^1, f_n^2 \in H^2(\R)$ be extensions of $f_n$ converging to $g_1, g_2$ 
in $H^r(\R)$ for $r<s$, see Lemma~\ref{lem:ext} below. Also take a   sequence $h_n\in H^2(\R^+)$   converging to $h$ in $H^{\frac{2s+1}{4}}(\R^+)$. By the uniqueness of mild $H^2(\R^+)$ solutions the  restriction of the corresponding solutions $u_n^1$, $u_n^2$ to $\R^+$ are the same. Since, by the fixed point argument,  the solutions $u^1$,
$u^2$ are the limits of $u_n^1$, $u_n^2$, respectively, in $H^r(\R)$, their restriction to $\R^+$ are the same.

\begin{lemma}\label{lem:ext} Fix $0<s<\frac12$. 
Let $g\in H^s(\R^+)$, $f\in H^2(\R^+)$, and let $g_e$ be an $H^s$ extension of $g$ to $\R$. Then there is an $H^2$ extension $f_e$ of $f$ to $\R$ so that 
$$\|g_e-f_e\|_{H^r(\R)}\les \|g-f\|_{H^s(\R^+)},\,\,\text{ for } r<s.
$$ 
\end{lemma}
\begin{proof}
Fix $0<s<\frac12$. We start with the following  \\
Claim. Fix $\psi \in H^s(\R)$ supported in $(-\infty,0]$. For any $\epsilon>0$, there is a function $\phi\in H^2(\R)$ supported in $(-\infty,0)$ such that $\|\phi-\psi\|_{H^r(\R)}<\epsilon$ for $r<s$.\\
To prove this claim first note that 
 $\chi_{(-\infty,-\delta)} \psi \to \psi $ 
in $L^2(\R)$ as $\delta \to 0^+$. Also note  by Lemma~\ref{lem:Hs0} that $\| \chi_{(-\infty,-\delta)} \psi\|_{H^s(\R)}\les \|   \psi\|_{H^s(\R)}$ uniformly in $\delta$. Therefore
 $\chi_{(-\infty,-\delta)} \psi \to \psi $ in $H^r$ for $r<s$ by interpolation. The claim follows by taking a smooth approximate identity $k_n$ supported in $(-\delta,\delta)$ for sufficiently small $\delta$, and letting
$\phi= [\chi_{(-\infty,-\delta)} \psi ] * k_n$ for sufficiently large $n$. 

To obtain the lemma from this claim, let $\widetilde f$ be an $H^2$ extension of $f$ to $\R$, and let $h$ be an $H^s$ extension of $g-f$ to $\R$ with $\|h\|_{H^s(\R)}\les \|g-f\|_{H^s(\R^+)}$. Apply the claim to $\psi=g_e-\widetilde f-h$ with $\epsilon=\|g-f\|_{H^s(\R^+)}$. Letting $f_e=\widetilde f +\phi$ yields the claim.  
\end{proof}
 
Finally, we prove that, for $s\in(0,1/2)$, the equation \eqref{nls} has at most one mild solution. To see this let $v =\lim v_n $  be a mild solution with initial and boundary data $g, h$. Let $g_e$ be an extension of $g$, and $u$ be the   solution we constructed. By the lemma above, we can extend $v_n(x,0)$ to $\R$ so that $v_n(\cdot,0)\in H^2(\R)$ and $v_n(\cdot,0)\to g_e$ in $H^r(\R)$, $r<s$. Let $u_n$ be the $H^2(\R)$ solution we constructed with initial data $v_n(\cdot,0)$ and boundary data 
$v_n(0,\cdot)$. By continuous dependence on initial data $u_n\to u$ in $H^r(\R)$,  and by  uniqueness in $H^2$ level, $u_n\big|_{\R^+}=v_n\big|_{\R^+}$. Therefore $u\big|_{\R^+}=v$.

\section{Proofs of Theorem~\ref{thm:smooth} and Theorem~\ref{thm:growth}} \label{sec:prt1t2}

\begin{proof}[Proof of Theorem~\ref{thm:smooth}]
Note that by \eqref{eq:gamma}, we have for $t\in [0,T]$
$$
u-W_0^t(g,h)=\eta(t) \int_0^tW_\R(t- t^\prime)  \eta(t^\prime /T) | u|^2  u\,d t^\prime - \eta(t) W_0^t (0,  q )(t),
$$
where
\begin{align*}
 q(t)=\eta(t ) D_0\Big(\int_0^tW_\R(t- t^\prime)  \eta(t^\prime /T) | u|^2  u \, d t^\prime \Big).
\end{align*}
Therefore, by the embedding  $X^{s,\frac12+}\subset C^0_tH^s_x$, the inequality \eqref{eq:xs2}, Lemma~\ref{lem:wbcont}, and Proposition~\ref{prop:duhamelkato}, we have
\begin{multline*}
\|u-W_0^t(g,h)\|_{C^0_{t\in [0,T]}H^{s+a}} \les \| \eta | u|^2  u \|_{X^{s+a,-\frac12+}}+  \|q\|_{H^{\frac{2(s+a)+1}{4}}_t}
 \\ \les  \| \eta | u|^2  u \|_{X^{s+a,-\frac12+}} + \left\{ \begin{array}{ll} \|\eta |u|^2 u\|_{X^{s+a,-\frac12+}} & \text{ for } 0\leq s+a \leq \frac12, \\
\|\eta |u|^2u\|_{X^{\frac12,\frac{2s+2a-3+}{4}}}+ \|\eta |u|^2 u\|_{X^{s+a,-\frac12+}}  & \text{ for }  \frac12 \leq s+a \leq \frac52. \end{array} \right.
\end{multline*}
Proposition~\ref{prop:smooth}, Proposition~\ref{prop:smooth2}, Theorem~\ref{thm:local}, and the local theory imply that
$$
\|u-W_0^t(g,h)\|_{C^0_{t\in [0,T]}H^{s+a}_{x\in\R^+}}\les \|u\|_{X^{s,b}}^3 \les (\|g\|_{H^s(\R^+)}+\|h\|_{H^{\frac{2s+1}{4}}(\R^+)})^3,
$$
which yields the claim.
\end{proof}

\begin{proof}[Proof of Theorem~\ref{thm:growth}]

First recall from Remark~\ref{rmk:Hslinbound} that $\|W_0^t(g,0)\|_{H^s(\R^+)}\les \|g\|_{H^s(\R^+)}$. Fix $T>0$.  Assuming the growth bound $\|u\|_{H^s(\R^+)}\leq f(T)$ where $f$ depends on $\|g\|_{H^s(\R^+)}$ and
$\|h\|_{ H^{s_1}(\R^+)}$, for some $s_1\geq \frac{2s+1}4$, let $\delta$ be the local existence time based on $f(T)$. Note that $\delta \approx (C+ f(T))^{-4} $, where $C=C(\|h\|_{H^{\frac{2s+1}{4}}(\R^+)})$. For $J\approx T/\delta$, we have
\begin{multline*}
\|u(J\delta)-W_0^{J\delta}(g,h)\|_{H^{s+a}(\R^+)} = \Big\| \sum_{k=1}^J W_{k\delta}^{J\delta} (u(k\delta),h)- W_{(k-1)\delta}^{J\delta} (u((k-1)\delta),h)\Big\|_{H^{s+a}(\R^+)}\\
\leq   \sum_{k=1}^J \big\|W_{k\delta}^{J\delta} (u(k\delta),h)- W_{(k-1)\delta}^{J\delta} (u((k-1)\delta),h)\big\|_{H^{s+a}(\R^+)} \\ \leq
 \sum_{k=1}^J \Big\|W_{k\delta}^{J\delta} \Big( \big[u(k\delta)-W_{(k-1)\delta}^{k\delta} (u((k-1)\delta),h)\big] ,0\Big)\Big\|_{H^{s+a}(\R^+)} \\
 \leq \sum_{k=1}^J \big\| u(k\delta)-W_{(k-1)\delta}^{k\delta} (u((k-1)\delta),h)  \big\|_{H^{s+a}(\R^+)} \les J f(t)^3\les \la T\ra f^7(T),
\end{multline*}
where the implicit constant depends only on $\|h\|_{H^{\frac{2s+1}{4}}(\R^+)}$. Here we used Remark~\ref{rmk:Hslinbound} in the second and third inequalities.

Recall that
$$
W_0^T(g,h)=W_\R(T) g_e +W_0^T(0,h-p),
$$
where $p(t)=\eta(t/\la T\ra) D_0(W_\R(t)g_e)$. Therefore, by Lemma~\ref{lem:Hs0}, we have
$$
\|W_0^T(g,h)\|_{H^s}\les \|g_e\|_{H^{s}} + \|(h-p)\
\chi_{(0,\infty)}\|_{H^{\frac{2s+1}4}} \les \|g\|_{H^{s}(\R^+)}+\|h\|_{H^{\frac{2s+1}4}(\R^+)} +\|p\|_{H^{\frac{2s+1}4}}.
$$
By applying Lemma~\ref{lem:kato} after we write $\eta(t/\la T\ra)=\sum_{k=1}^{\la T\ra}\eta_k(t)$, we obtain
$$
\|W_0^T(g,h)\|_{H^s(\R^+)}\les \la T\ra \|g\|_{H^{s}(\R^+)} + \|h\|_{H^{\frac{2s+1}4}(\R^+)}.
$$
Using this for $s+a$ and the previous calculation, we obtain
$$
\|u(T)\|_{H^{s+a}(\R^+)}\les \la T\ra \|g\|_{H^{s+a}(\R^+)} + \|h\|_{H^{\frac{2s+2a+1}4}(\R^+)}+\la T\ra f^7(T),
$$
where the implicit constant depends on $\|h\|_{H^{\frac{2s+1}{4}}(\R^+)}$.
In the defocusing case, for $s=s_1=1$, $f(t)\approx 1$, whereas in the focusing case $f$ grows exponentially. This implies that
in the defocusing case
$$
\|u(t)\|_{H^s(\R^+)}\les \left\{\begin{array}{ll}\la T\ra, & 1<s<\frac32,\\ \la T\ra^8, & \frac32<s<2,\\ \la T\ra^{57}, & 2<s<\frac52,\end{array}\right.
$$

\end{proof}

\section{Appendix}\label{sec:appendix}
We start with a slight improvement of the energy bound from \cite{bonaetal}.  We note that the Gagliardo-Nirenberg inequality in one dimension implies that
$$
\|f\|_{L^4(\R^+)}^4\leq C \|f\|_{L^2(\R^+)}^3 \|f_x\|_{L^2(\R^+)}.
$$
\begin{proposition}\label{prop:energy}
We have the following a priori estimates for the solutions of \eqref{nls}.
When $\lambda=-1$,
$$
\|u\|_{H^1}\leq C_{\|g\|_{H^1},\|h\|_{H^1}}.
$$
When $\lambda=1$,
$$
\|u\|_{H^1}\leq C e^{Dt},
$$
where $C=C(\|g\|_{H^1},\|h\|_{H^1}),$ and  $D=D(\|h\|_{H^1})$.
\end{proposition}
\begin{proof}
We present the proof for the focusing case,
\begin{align}\label{focnls}
&iu_t+u_{xx}+ |u|^2u=0,\,\,\,\,x\in\R^+, t\in \R^+,\\
&u(x,0)=g(x), \,\,\,\,u(0,t)=h(t), \nn
\end{align}
the defocusing case is easier, see below. In what follows we drop $\R^+$ from $\|\cdot\|_{H^s(\R^+)}$ notation. 

The following identities can be justified by approximation by $H^2$ solutions:
\begin{align}
\label{eq:m} &\partial_t|u|^2= -2 \Im(u_x\overline{u})_x ,\\
\label{eq:e}&\partial_t(|u_x|^2-\frac12 |u|^4)=2\Re(u_x\overline{u_t})_x,\\
\label{eq:l}&\partial_x(|u_x|^2+\frac12 |u|^4)=-i\big[(u\overline{u_x})_t-(u\overline{u_t})_x\big].
\end{align}
We start by estimating $\|u\|_{L^2}$. By integrating \eqref{eq:m}   in $ [0,\infty)\times [0,t]$, we obtain
$$
\int_0^\infty |u(x,t)|^2 dx = \int_0^\infty |g(x )|^2 dx+2\Im \int_0^t u_x(0,s) \overline{h(s)} ds.
$$
By Cauchy-Schwarz inequality, we have
$$
\|u\|_{L^2_x}^2\leq \|g\|_{L^2}^2 + 2\|h\|_{L^2_{[0,t]}} \Big[\int_0^t |u_x(0,s)|^2ds\Big]^{1/2}.
$$
This implies by Sobolev embedding on $h$ that
\begin{align}\label{eq:ul21}
\|u\|_{L^2_x} \les \|g\|_{L^2} + t^{1/4} \Big[\int_0^t |u_x(0,s)|^2ds\Big]^{1/4},
\end{align}
where the implicit constant depends only on $\|h\|_{H^1}$.

To estimate $\|u_x\|_{L^2}$, we integrate \eqref{eq:e} in $ [0,\infty)\times [0,t]$:
\begin{multline*}
\int_0^\infty |u_x(x,t)|^2 dx = \int_0^\infty |g^\prime(x )|^2 dx \\ +\frac12\int_0^\infty |u(x,t)|^4 dx - \frac12\int_0^\infty |g(x)|^4 dx
-2\Re \int_0^t u_x(0,s) \overline{h^\prime(s)} ds.
\end{multline*}
Using Gagliardo-Nirenberg and Cauchy-Schwarz inequalities we have
$$
\|u_x\|_{L^2_x}^2 \leq \|g\|_{H^1}^2 + C \|u\|_{L^2_x}^3\|u_x\|_{L^2_x}+2\|h\|_{H^1} \Big[\int_0^t |u_x(0,s)|^2ds\Big]^{1/2}.
$$
Note that $y^2\leq A^2+By$ implies that $y\les A+B$. Using this and then   \eqref{eq:ul21}, we obtain
\begin{multline}\label{eq:uxl21}
\|u_x\|_{L^2_x} \les \|g\|_{H^1}  +  \|u\|_{L^2_x}^3 +  \Big[\int_0^t |u_x(0,s)|^2ds\Big]^{1/4}\\
\les  \|g\|_{H^1}  +  \|g\|_{L^2}^3 + t^{3/4} \Big[\int_0^t |u_x(0,s)|^2ds\Big]^{3/4} +  \Big[\int_0^t |u_x(0,s)|^2ds\Big]^{1/4}.
\end{multline}
Finally, to estimate $\int_0^t |u_x(0,s)|^2ds$, we  integrate \eqref{eq:l} in $x $ from 0 to $\infty$:
$$
|u_x(0,t)|^2+\frac12 |u(0,t)|^4=i \frac{d}{dt} \int_0^\infty u\overline{u_x} dx +i u(0,t)\overline{u_t(0,t)}.
$$
Integrating this in $[0,t]$ we have
\begin{multline*}
\int_0^t |u_x(0,s)|^2ds =-\frac12\int_0^t |h(s)|^4ds \\ + i \int_0^\infty u(x,t)\overline{u_x(x,t)} dx -  i \int_0^\infty g(x)\overline{g^\prime(x)} dx +i\int_0^t
h(s)\overline{h^\prime(s)} ds.
\end{multline*}
And hence
\begin{align}
\label{eq:ux0} I:=\int_0^t |u_x(0,s)|^2ds\leq \|u\|_{L^2_x}\|u_x\|_{L^2_x}+ \|g\|_{L^2}\|g\|_{H^1}+ \|h\|_{L^2}\|h\|_{H^1}.
\end{align}
Using \eqref{eq:ul21} and \eqref{eq:uxl21} in \eqref{eq:ux0} we have
\begin{multline*}
I \les tI + t^{1/4}I^{1/2}+  t^{1/4} I^{1/4} (\|g\|_{H^1} +   \|g\|_{L^2}^3) +\|g\|_{L^2}   ( t^{3/4} I^{3/4} +  I^{1/4})\\ + \|g\|_{L^2}^4+   \|g\|_{L^2}\|g\|_{H^1}+ 1
\end{multline*}
Taking $t$ small (depending only on $\|h\|_{H^1}$), we get
$$
I \les  I^{1/4} (\|g\|_{H^1} +   \|g\|_{L^2}^3) +   I^{3/4}\|g\|_{L^2}    + \|g\|_{L^2}^4+   \|g\|_{L^2}\|g\|_{H^1}+ 1.
$$
This implies that
$$
I \les \|g\|_{H^1}^{4/3} +  \|g\|_{L^2}^4    + 1.
$$
Using this in \eqref{eq:ul21} and \eqref{eq:uxl21}, we have 
$$
\|u\|_{L^2_x}\les   \|g\|_{L^2} +   \|g\|_{H^1}^{1/3} +1,
$$
$$
\|u_x\|_{L^2_x}  \les   \|g\|_{H^1} +   \|g\|_{L^2}^3 + 1.
$$
This implies that
$$
\|u_x\|_{L^2_x} +\|u\|_{L^2_x}^3\les   \|g\|_{H^1} + \|g\|_{L^2}^3+1.
$$
Iterating this bound implies that
$$
\|u_x\|_{L^2_x} +\|u\|_{L^2_x}^3\leq C  e^{Dt},
$$
where $D$ depends only on $\|h\|_{H^1}$.

In the defocusing case  we don't have the term coming from the Gagliardo Nirenberg inequality.  We instead have the following inequalities
$$
\|u\|_{L^2}\les 1+I^{1/4},
$$
$$
\|u_x\|_{L^2}\les 1+I^{1/4},
$$
$$
I\les 1+\|u\|_{H^1}^2,
$$
where the implicit constants depend on $\|g\|_{H^1}$ and $\|h\|_{H^1}$. Therefore
$$
\|u\|_{H^1}\les 1+ \|u\|_{H^1}^{1/2},
$$
which implies that $\|u\|_{H^1}$ remains bounded.
\end{proof}

Finally, we have the following lemmas. For a proof of the former, see \cite{et3}.
\begin{lemma}\label{lem:sums} If $\beta \geq \gamma \geq 0$ and $\beta + \gamma > 1$, then
\[ \int \frac{1}{\langle x-a_1 \rangle ^\beta \langle x - a_2 \rangle^\gamma} dx \lesssim \langle a_1-a_2 \rangle ^{-\gamma} \phi_\beta(a_1-a_2),  \]
where
\[ \phi_\beta(a) \sim \begin{cases} 1 & \beta > 1 \\  \log(1 + \langle a \rangle ) & \beta = 1 \\ \langle a \rangle ^{1-\beta} &\beta < 1 .\end{cases} \]
\end{lemma}

\begin{lemma}\label{lem:sqrt} For fixed $\rho \in (\frac12, 1)$, we have
\[ \int \frac{1}{\langle x \rangle ^{ \rho} \sqrt{|x - a|}} dx \lesssim \frac{1}{\la a\ra^{ \rho-\frac12 }}. \]
\end{lemma}
\begin{proof} Let $A=\{x:|x-a|>1\}$, and $B=\{x:|x-a|\leq 1\}$.
Note that
$$
\int_B \frac{1}{\langle x \rangle ^{ \rho} \sqrt{|x - a|}} dx \les \frac1{\la a \ra^\rho} \int_B  \frac1{\sqrt{|x - a|}} dx \les \frac1{\la a \ra^\rho}.
$$
Finally, using Lemma~\ref{lem:sums}, we have
$$
 \int_A \frac1{\la x \ra^{ \rho} \sqrt{|x - a|}} dx \les  \int_A \frac1{\la x \ra^\rho \sqrt{\la x-a\ra}}   dx \les \frac1{\la a \ra^{\rho-\frac12}}.
$$
\end{proof}

\end{document}